\documentclass[ECP]{ejpecp} % replace ECP by EJP if needed.  

\newcommand\GT{\mathbb{G}\mathbb{T}}
\newcommand\CC{\mathbb{C}}

\newcommand\RR{{\mathbb R}}

\SHORTTITLE{Asymptotics of Schur functions on almost staircase partitions} 

\TITLE{Asymptotics of Schur functions on almost staircase partitions} % \thanks is optional. Insert line breaks with \\

\AUTHORS{%
  Zhongyang Li\footnote{University of Connecticut, United States of America.
    \EMAIL{zhongyang.li@uconn.edu}}}%AUTHORS
 % Separate items with ;

\AMSSUBJ{80B20} % Edit. Separate items with ;
\SUBMITTED{May 5, 2020} % Edit.
\ACCEPTED{July 2, 2020} % Edit.

\VOLUME{0}
\YEAR{}
\PAPERNUM{0}
\DOI{vVOL-PID}

\ABSTRACT{We study the asymptotics of Schur polynomials with partitions $\lambda$ which are almost staircase; more precisely, partitions that differ from $((m-1)(N-1),(m-1)(N-2),\ldots,(m-1),0)$ by at most one component at the beginning as $N\rightarrow \infty$, for a positive integer $m\ge 1$ independent of $N$. By applying either determinant formulas or integral representations for Schur functions, we show that $\frac{1}{N}\log \frac{s_{\lambda}(u_1,\ldots,u_k, x_{k+1},\ldots,x_N)}{s_{\lambda}(x_1,\ldots,x_N)}$ converges to a sum of $k$ single-variable holomorphic functions, each of which  depends on the variable $u_i$ for $1\leq i\leq k$, when there are only finitely many distinct $x_i$'s and each $u_i$ is in a neighborhood of $x_i$, as $N\rightarrow\infty$. The results  are related to the law of large numbers and central limit theorem for the dimer configurations on contracting square-hexagon lattices with certain boundary conditions.}

\begin{document}

%%%%%%%%%%%%%%%%%%%%%%%%%%%%%%%%%%%%%%%%%%%%%%%%%%%%%%%%%%%%%%%%%%%
%%                                                               %%
%% No need for \maketitle.                                       %%
%%                                                               %%
%%%%%%%%%%%%%%%%%%%%%%%%%%%%%%%%%%%%%%%%%%%%%%%%%%%%%%%%%%%%%%%%%%%

%%%%%%%%%%%%%%%%%%%%%%%%%%%%%%%%%%%%%%%%%%%%%%%%%%%%%%%%%%%%%%%%%%%
%%                                                               %%
%% Please replace what follows by the body of your article       %%
%% (up to the bibliography):                                     %%
%%                                                               %%
%%%%%%%%%%%%%%%%%%%%%%%%%%%%%%%%%%%%%%%%%%%%%%%%%%%%%%%%%%%%%%%%%%%

\section{Introduction}

Schur polynomials, named after Issai Schur, are a class of symmetric polynomials indexed by decreasing sequences of non-negative integers (partitions), which form a linear basis for the space of all symmetric polynomials; see \cite{IGM15}. Besides their applications in representation theory, Schur polynomials also play an important role in the study of integrable lattice models in statistical mechanics (see \cite{AB07,AB11}). One example of such a model is the dimer model, or equivalently, the random tiling model; see \cite{CKP,KO}. In this paper, we study the asymptotics of Schur polynomials on partitions which are almost periodic; the results are related the law of large numbers and central limit theorem for dimer configurations on contracting square-hexagon lattices. The connection between asymptotics of Schur polynomials and scaling limit of random tilings has been investigated, see \cite{GP15,bg,bg16, BG17} for uniform perfect matchings on the hexagon lattice (random lozenge tiling); \cite{bk} for uniform perfect matchings on the square grid (random domino tiling); and \cite{BL17,ZL18,Li182} for periodically weighted perfect matchings on the square-hexagon lattice. This paper further develops the technique in \cite{BL17,ZL18,Li182} to study the asymptotics of Schur polynomials on more general partitions. We shall begin with the definition of Schur functions.

\subsection{Partitions, Young diagrams and Schur functions}

We denote by $\GT_N^+$ the set of $N$-tuples $\lambda$ of nonnegative integers satisfying $\lambda_1\geq \lambda_2\ldots\geq \lambda_N\geq 0$. For $\lambda\in \GT_N^+$, let
\begin{eqnarray*}
|\lambda|:=\sum_{i=1}^{N}\lambda_i.
\end{eqnarray*}
 
 A graphic way to represent a non-negative signature $\mu$ is through its
\emph{Young diagram} $Y_\lambda$, a collection of $|\lambda|$ boxes arranged on
non-increasing rows aligned on the left: with
$\lambda_1$ boxes on the first row, $\lambda_2$ boxes on the second row, \dots, $\lambda_N$
boxes on the $N$th row.  Note that elements in $\GT_N^+$ are in bijection with all the Young diagrams with $N$ rows (rows are allowed to have zero length).

\subsection{Main Results}
Before stating the main theorem concerning asymptotics of Schur polynomials in this paper, we first introduce a few definitions, notation and assumptions.
For any positive integer $a$, let
\begin{eqnarray*}
[a]=\{1,2,\ldots,a\}.
\end{eqnarray*}

Let 
\begin{eqnarray}
X&=&(x_1,\ldots,x_N).\label{xnn}
\end{eqnarray}
and
\begin{eqnarray*}
W&=&(w_1,\ldots,w_N),
\end{eqnarray*}
where
\begin{eqnarray}
w_i&=&\begin{cases}u_i,\ \mathrm{if}\ 1\leq i\leq k;\\ x_i,\ \mathrm{if}\ k+1\leq i\leq N,\end{cases}\label{wi}
\end{eqnarray}
where $k$ is a fixed positive integer independent of $N$. 

Of special interest is when a lot of values in $x_1,\ldots,x_N$ are equal, in particular, when there are only finitely many distinct values in $x_1,\ldots,x_N$ as $N\rightarrow\infty$. The asymptotics of Schur polynomials $s_{\lambda}(x_1,\ldots,x_N)$ as $N\rightarrow\infty$ in this case are related to the limit shape and height fluctuations of perfect matchings on a square-hexagon lattice with periodic weights; see \cite{BL17,Li182,ZL18}. Let $n$ be a positive integer which is fixed as $N\rightarrow\infty$. We may make the following assumption.

\begin{assumption}\label{ap1} Let $X$ be given by (\ref{xnn}), and assume that in $x_1,\ldots,x_N$, there are only finitely many different values as $N\rightarrow\infty$.
Let $x_1> x_2>\ldots>x_n>0$ be all the distinct values in $X$. For $1\leq j\leq n$, let
\begin{eqnarray*}
K_j^{(N)}=|\{i:1\leq i\leq N, x_i=x_j\}|.
\end{eqnarray*}
Then
\begin{eqnarray}
\lim_{N\rightarrow\infty}\frac{K_j^{(N)}}{N}=\gamma_j\in[0,1].\label{adj}
\end{eqnarray}
such that $\sum_{j=1}^{n}\gamma_j=1$.
\end{assumption}

In Assumption \ref{ap1}, for $1\leq j\leq n$, we assume the asymptotical density of $x_j$ in $x_1,\ldots,x_N$ is a constant $\gamma_j$. In the periodic case when 
\begin{eqnarray*}
x_{l}=x_{s}, \forall\ l,s\in[N],\ [(l-s)\mod n]=0,
\end{eqnarray*}
it is straightforward to check that (\ref{adj}) holds with $\gamma_j=\frac{1}{n}$ for all $1\leq j\leq n$.

The main theorem proved in this paper is the following.

\begin{theorem}\label{t17}Let $\lambda:=(\lambda_1,\ldots,\lambda_N)\in \GT_N^+$. Let $m$ be a fixed finite positive integer independent of $N$ as $N\rightarrow\infty$.
Assume
\begin{enumerate}
\item $\lambda_1=\alpha_1 N+O(1)$, where $\alpha_1>1$ is a fixed positive number independent of $N$.
\item For all the $2\leq j\leq N$, $\lambda_j=(m-1)(N-j)$; and
\item Assumption \ref{ap1} holds.  
 \end{enumerate}
 Then
\begin{eqnarray}
\lim_{N\rightarrow\infty}\frac{1}{N}\log \frac{s_{\lambda}(W)}{s_{\lambda}(X)}=\sum_{1\leq i\leq k}[Q(u_i)-Q(x_i)],\label{mf}
\end{eqnarray}
where 
\begin{eqnarray}
Q(u)=\sum_{1\leq j\leq n}\gamma_j\log\left[\frac{u^m-x_j^m}{u-x_j}\right].\label{qu}
\end{eqnarray}
Moreover, the convergence in (\ref{mf}) is uniform when each $u_i$ is in a compact complex neighborhood of $x_i$.
\end{theorem}

Let $\lambda^{(m)}$ be the staircase partition given by
\begin{eqnarray}
\lambda^{(m)}=((m-1)(N-1),(m-1)(N-2),\ldots, m-1,0).\label{sc}
\end{eqnarray}
 In this case, the asymptotics of Schur functions can be obtained by the following explicit formula (see example 1.3.7 of \cite{IGM15})
\begin{eqnarray}
s_{\lambda^{(m)}}(x_1,\ldots,x_N)=\prod_{1\leq i<j\leq N}\frac{x_i^m-x_j^m}{x_i-x_j};\label{ssf}
\end{eqnarray}
and the asymptotics $\lim_{N\rightarrow\infty}\frac{1}{N}\log \frac{s_{\lambda^{(m)}}(W)}{s_{\lambda^{(m)}}(X)}$ can be shown to be exactly right side of (\ref{mf}) from the formula (\ref{ssf}) in a straightforward way.
Theorem \ref{t17} discusses the asymptotics of Schur polynomials on more general partitions; more precisely, the partition differs from the staircase (\ref{sc}) by at most finitely many entries at the beginning. To prove Theorem \ref{t17}, we shall split $\frac{1}{N}\log \frac{s_{\lambda}(W)}{s_{\lambda}(X)}$ into the sum of 3 terms, one of which is given by $\frac{1}{N}\log \frac{s_{\lambda^{(m)}}(W)}{s_{\lambda^{(m)}}(X)}$. Since it is known that  $\lim_{N\rightarrow\infty}\frac{1}{N}\log \frac{s_{\lambda^{(m)}}(W)}{s_{\lambda^{(m)}}(X)}$ is equal to the right hand side of (\ref{ssf}), it suffices to show that the limit of the sum of the other two terms, as $N\rightarrow \infty$, vanishes. Theorem \ref{t17} is proved in Sect.~\ref{pt17}.

\subsection{Applications}
Recall that a perfect matching, or a dimer configuration, on a graph is a subset of edges such that each vertex is incident to exactly one edge. We consider the probability measure on perfect matchings of a finite graph in which the probability of a perfect matching is proportional to the product of weights of present edges.
 The asymptotical formula for Schur functions given by Theorem \ref{t17} may be used to obtain the limit shapes and fluctuations of perfect matchings on the square-hexagon lattice with certain boundary conditions (contracting square-hexagon lattice), by the general arguments as described in \cite{BL17, ZL18, Li182, GS18}. More precisely, 
\begin{enumerate}
\item for the weighted perfect matching model on a contracting square-hexagon lattice, where the edge weights are periodic with period $1\times n$ and $n$ is an arbitrary fixed positive integer, the partition function (weighted sum of perfect matchings) can be computed by the Schur polynomial depending on edge weights and bottom boundary condition; see Proposition 2.18 of \cite{BL17}.
\item The derivatives of the Schur generating function give the moments of the counting measure for perfect matchings on each row. 
\item The uniform convergence result given by Theorem \ref{t17} also guarantees the convergence of the derivatives, hence the limit of the moments of the counting measure of perfect matchings on each row can be obtained, which gives the limit shape. 
\item To consider the fluctuations around the limit shape, one needs to check the Wick's formula for the moments to obtain the Gaussian fluctuation.
\end{enumerate}
Following the procedure above, we can prove that when the boundary partition differs from the staircase partition by at most one component, the limit shape and limit height fluctuations are the same as those for the staircase partition on the boundary; the latter was studied in \cite{BL17,ZL18}. 

\begin{theorem}\label{ad}Consider a contracting square-hexagon lattice consisting of $2N+1$ rows of vertices as defined in Definition 2.3 of \cite{BL17}, and let $I_2$ be defined as in Definition 2.4 of \cite{BL17}. Assume that the edge weights are assigned periodically with period $n$, i.e. $x_i=x_j$ and $y_i=y_j$ if $(i-j)\mod n=0$, where $n$ is a fixed positive integer independent of $N$. Assume that the boundary partition is given by $(\lambda_1,(m-1)(N-2),(m-1)(N-3),\ldots,(m-1),0)\in \GT_N$ and satisfy the assumptions of Theorem \ref{t17}. For $\kappa\in (0,1)$, let $\mathbf{m}^{\kappa}$ be the weak limit of the counting measure for partitions corresponding to dimer configurations at the level $\kappa$ as $N\rightarrow\infty$, where assume that the bottom boundary is at level 0, while the top boundary is at level 1. Then 
 \begin{equation*}
\int_{\RR}x^{p}\mathbf{m}^{\kappa}(dx)=
%=\sum_{i=1}^{n}\frac{1}{2(p+1)\pi \mathbf{i}}\oint_{x_{t+i}}\frac{dz}{z}\left(zQ_{\kappa}'(z)+\sum_{j=1}^{n}\frac{z}{n(z-x_{j})}\right)^{p+1}
\frac{1}{2(p+1)\pi \mathbf{i}}\oint_{C_{x_1,\ldots,x_n}}\frac{dz}{z}\left(zQ_{\kappa}'(z)+\sum_{j=1}^{n}\frac{z}{n(z-x_{j})}\right)^{p+1},
\end{equation*}
where $C_{x_1,\ldots,x_n}$ is a simple, closed, positively oriented, contour containing only the poles $x_1,\ldots,x_n$ of the integrand, and no other singularities; and
\begin{eqnarray*}
Q'_{\kappa}(z)=\frac{1}{n(1-\kappa)}\sum_{1\leq j\leq n}\left(\frac{mu^{m-1}}{z^m-x_j^m}-\frac{1}{z-x_j}\right)+\frac{\kappa}{n(1-\kappa)}\sum_{i\in\{1,2,\ldots,n\}\cap I_2}\frac{y_i}{1+y_iz}.
\end{eqnarray*}
\end{theorem}

The proof of Theorem \ref{ad} follows from Theorem \ref{t17} and the same arguments as in Section 8.2 of \cite{BL17}.

The effects of changing boundary conditions to the distribution of dimer configurations have been studied, see, for example, \cite{CKP} for the case of limit shape of perfect matchings on square grid via a variational principle. In particular, it is proved in \cite{CKP}  that if the rescaled (by a $\frac{1}{N}$ multiple) boundary height function has a fixed limit as $N\rightarrow\infty$, then the limit shape is uniquely determined.  See \cite{LB19} for the case that the local statistics of uniform perfect matchings on the hexagonal lattice are preserved under small perturbation of boundary heights. The paper \cite{LB19} discusses the distribution of uniform lozenge tilings near an interior point of the domain, yet the global limit shape is not discussed.  In our assumption, when the boundary condition satisfies Condition (1) in Theorem \ref{t17}, the rescaled region (by an $\frac{1}{N}$ multiple) of the tiling is different, indeed larger even in the scaling limit, from the region of the tiling when the boundary condition is staircase. Theorem \ref{ad} shows that the weak limit of the counting measure at each level does not change.

\section{Proof of Theorem \ref{t17}}\label{pt17}

In this section, we prove Theorem \ref{t17}. 

Assume that $m$ is a fixed positive integer independent of $N$. Let $\lambda^{(m)}\in \GT_N^+$ be the staircase partition of length $N$ as given by (\ref{sc}), and let 
\begin{eqnarray*}
\lambda=(\lambda_1,\ldots,\lambda_N)\in \GT_N^+.
\end{eqnarray*}
be a length-$N$ partition such that the entries of $\lambda$ differ from those of $\lambda^{(m)}$ by at most $\ell$ entries $\lambda_1,\ldots,\lambda_l$ at the beginning, where $\ell$ is a fixed positive integer independent of $N$.

 We consider the asymptotics of
\begin{eqnarray}
\frac{1}{N}\log \frac{s_{\lambda}(W)}{s_{\lambda}(X)}.\label{an}
\end{eqnarray}
as $N\rightarrow\infty$. Note that
\begin{eqnarray*}
\frac{s_{\lambda}(W)}{s_{\lambda}(X)}=S_1\cdot S_2\cdot S_3,
\end{eqnarray*}
where
\begin{eqnarray*}
S_1=\frac{s_{\lambda}(W)}{s_{\lambda^{(m)}}(W)};\qquad
S_2=\frac{s_{\lambda^{(m)}}(W)}{s_{\lambda^{(m)}}(X)};\qquad
S_3=\frac{s_{\lambda^{(m)}}(X)}{s_{\lambda}(X)}.
\end{eqnarray*}
\begin{lemma}\label{l31}Suppose that Assumption \ref{ap1} holds.  Let $W$ be given by (\ref{wi}). Then
\begin{eqnarray*}
\lim_{N\rightarrow\infty}\frac{1}{N}\log S_2=\sum_{1\leq i\leq k}[Q(u_i)-Q(x_i)];
\end{eqnarray*}
where $Q(u)$ is given by (\ref{qu}).
\end{lemma}
\begin{proof}By example 1.3.7 of \cite{IGM15}, we have
\begin{eqnarray*}
s_{\lambda^{(m)}}(X)=\prod_{1\leq i<j\leq N}\frac{x_i^m-x_j^m}{x_i-x_j}.
\end{eqnarray*}
Then by (\ref{wi}) we have
\begin{eqnarray*}
\frac{1}{N}\log S_2&=&\frac{1}{N}\sum_{1\leq i<j\leq N}\log\left(\frac{w_i^m-w_j^m}{x_i^m-x_j^m}\frac{x_i-x_j}{w_i-w_j}\right)\\
&=&\frac{1}{N}\sum_{1\leq i<j\leq k}\log\left(\frac{u_i^m-u_j^m}{x_i^m-x_j^m}\frac{x_i-x_j}{u_i-u_j}\right)+\frac{1}{N}\sum_{[1\leq i\leq k]}\sum_{[k+1\leq j\leq N]}\log\left(\frac{u_i^m-x_j^m}{x_i^m-x_j^m}\frac{x_i-x_j}{u_i-x_j}\right).
\end{eqnarray*}
When $k$ is fixed as $N\rightarrow\infty$, we have 
\begin{eqnarray*}
\lim_{N\rightarrow\infty}\sum_{1\leq i<j\leq k}\frac{1}{N}\log\left(\frac{u_i^m-u_j^m}{x_i^m-x_j^m}\frac{x_i-x_j}{u_i-u_j}\right)=0.
\end{eqnarray*}
Then the lemma follows from explicit computations.
\end{proof}

Hence to study the asymptotics of (\ref{an}), it suffices to study the asymptotics of $\frac{1}{N}\log S_1$ and  $\frac{1}{N}\log S_3$, as $N\rightarrow\infty$.
Let
\begin{eqnarray*}
\mu^{(m)}&=&((N-1)m,(N-2)m,\ldots,m,0).
\end{eqnarray*}
Let $\Delta(X)$ (resp.\ $\Delta(W)$) denote the Vandermonde determinant of the variable $X$ (resp.\ $W$). From the well-known formula to compute the Schur function, we obtain
\begin{eqnarray}
s_{\lambda}(W)&=&\frac{\det[e^{(\lambda_i+N-i)\log w_j}]_{1\leq i,j\leq N}}{\Delta(W)},\label{slw}\\
s_{\lambda^{(m)}}(W)&=&\frac{\det[e^{(m(N-i)\log w_j}]_{1\leq i,j\leq N}}{\Delta(W)},\label{slmw}\\
s_{\lambda}(X)&=&\frac{\det[e^{(\lambda_i+N-i)\log x_j}]_{1\leq i,j\leq N}}{\Delta(X)},\notag\\
s_{\lambda^{(m)}}(X)&=&\frac{\det[e^{m(N-i)\log x_j}]_{1\leq i,j\leq N}}{\Delta(X)}.\notag
\end{eqnarray}

\begin{lemma}\label{l62}Suppose that $\lambda\in\GT_N^+$ differs from $\lambda^{(m)}$ by at most $l$ components at the beginning, where $l\geq 1$ is a positive integer. Then
\begin{eqnarray*}
&&\frac{s_{\lambda}(W)}{s_{\lambda^{(m)}}(W)}\\
&=&\sum_{J=j_1<j_2<\ldots<j_l}\left(\prod_{r\in J}\frac{1}{\prod_{s\neq r}(w_r^m-w_s^m)}\right)\sum_{\sigma\in S_l}(-1)^{\sigma}\det\left[w_{j_t}^{m(i-1)+\lambda_{\sigma(t)}+N-\sigma(t)}\right]_{1\leq i,t\leq l}.
\end{eqnarray*}

\end{lemma}
\begin{proof}By (\ref{slw}), (\ref{slmw}) and Condition II of Theorem \ref{t17}, we obtain
\begin{eqnarray*}
&&\frac{s_{\lambda}(W)}{s_{\lambda^{(m)}}(W)}=\frac{\det[e^{(\lambda_i+N-i)\log w_j}]_{1\leq i,j\leq N}}{\det[e^{m(N-i)\log w_j}]_{1\leq i,j\leq N}}\\
&=&\sum_{J=j_1<j_2<\ldots<j_l}(-1)^{\sum_{t=1}^{l}(j_t+t)}\det[e^{(\lambda_t+N-t)\log w_{j_s}}]_{1\leq t\leq l,1\leq s\leq l}\frac{\det[e^{m(N-t)\log w_s}]_{l+1\leq t\leq N, s\in [N]\setminus J}}{\det[e^{m(N-i)\log w_j}]_{1\leq i,j\leq N}}.
\end{eqnarray*}
For each set $J\subset[N]$ with $|J|=\ell$ we obtain
\begin{eqnarray*}
\frac{\det[e^{m(N-t)\log w_s}]_{l+1\leq t\leq N, s\in [N]\setminus J}}{\det[e^{m(N-i)\log w_j}]_{1\leq i,j\leq N}}
&=&\frac{\prod_{i,j\in\{ [N]\setminus J\},i<j}(w_i^m-w_j^m)}{\prod_{1\leq i<j\leq N}(w_i^m-w_j^m)}\\
&=&\frac{1}{\left[\prod_{i\in J,j>i}(w_i^m-w_j^m)\right]\left[\prod_{j\in J, i\in [N]\setminus J,i<j}(w_i^m-w_j^m)\right]}\\
&=&\prod_{i<j,i,j\in J}(w_i^m-w_j^m)\prod_{r\in J}\frac{(-1)^{r-1}}{\prod_{s\neq r}(w_r^m-w_s^m)}.
\end{eqnarray*}
Moreover, let $S_l$ be the symmetric group of $l$ elements,
\begin{eqnarray*}
&&\left[\prod_{i<j,i,j\in J}(w_i^m-w_j^m)\right]\det[e^{(\lambda_t+N-t)\log w_{j_s}}]_{1\leq t\leq l,1\leq s\leq l}\\
&=&(-1)^{\frac{l(l-1)}{2}}\det\left[w_{j_t}^{m(i-1)}\right]_{1\leq i,t\leq l}\left(\sum_{\sigma\in S_l}(-1)^{\sigma}\prod_{t=1}^{l}w_{j_{\sigma(t)}}^{\lambda_t+N-t}\right)\\
&=&(-1)^{\frac{l(l-1)}{2}}\left[\sum_{\sigma\in S_l}(-1)^{\sigma}\det\left[w_{j_t}^{m(i-1)+\lambda_{\sigma(t)}+N-\sigma(t)}\right]_{1\leq i,t\leq l}\right].
\end{eqnarray*}
Then the lemma follows.
\end{proof}

Now we discuss the asymptotics of $\frac{1}{N}\log\frac{s_{\lambda}(W)}{s_{\lambda^{(m)}}(W)}$ as $N\rightarrow\infty$, when $\lambda$ differs from $\lambda^{(m)}$ only in the first component. The proof is inspired by \cite{GP15}.
 Let $\mathbf{i}$ be the imaginary unit satisfying $\mathbf{i}^2=-1$.

\begin{lemma}\label{l63}Assume the assumptions of Lemma \ref{l62} hold with $l=1$. Then
\begin{eqnarray}
\frac{s_{\lambda}(W)}{s_{\lambda^{(m)}}(W)}=
\frac{1}{2\pi\mathbf{i}}\oint_{C} \frac{mz^{\lambda_1+N-1+m-1}}{\prod_{i=1}^{N}(z^m-w_i^m)}dz.\label{ig}
\end{eqnarray}
Here the contour $C$ encloses only the poles of the integrand at $z=w_i$, $i=1,\ldots,N$.
\end{lemma}
\begin{proof}By Lemma \ref{l62}, when $l=1$ we obtain
\begin{eqnarray}
&&\frac{s_{\lambda}(W)}{s_{\lambda^{(m)}}(W)}
=\sum_{j=1}^{N}\frac{w_{j}^{\lambda_1+N-1}}{\prod_{s\neq j}(w_j^m-w_s^m)}.\label{se}
\end{eqnarray}
The right hand side of (\ref{se}) is exactly the sum of residues of the integrand of (\ref{ig}) at all the poles $z=w_i$. Then the lemma follows.
\end{proof}

Let $\xi=z^m$ and 
\begin{eqnarray}
y=\frac{\lambda_1+N-1}{mN}. \label{dy}
\end{eqnarray}
Then (\ref{ig}) becomes
\begin{eqnarray}
\frac{s_{\lambda}(W)}{s_{\lambda^{(m)}}(W)}=
\frac{1}{2\pi\mathbf{i}}\oint_{C} \frac{\xi^{Ny}}{\prod_{i=1}^{N}(\xi-w_i^m)}d\xi,\label{sds}
\end{eqnarray}
where the contour $C$ encloses all the singularities of the integrand.
The goal is to analyze the asymptotics of (\ref{sds}). 

Suppose that Assumption \ref{ap1} holds. Define a piecewise continuous, decreasing function $f: [0,1)\rightarrow [x_n^m,x_1^m]$ as follows
\begin{eqnarray}
f(y)=x_i^m,\ \mathrm{if}\ y\in \left[\sum_{j=1}^{i-1}\gamma_j,\sum_{j=1}^{i}\gamma_j\right),\label{fy}
\end{eqnarray}
where $i\in [n]$. Then it is straightforward to check the following lemma concerning $f$:

\begin{lemma}\label{ap66}Let $\sigma$ be a permutation of $[N]$ such that
\begin{eqnarray*}
x_{\sigma(1)}\geq x_{\sigma(2)}\geq\ldots\geq x_{\sigma(N)}.
\end{eqnarray*}
Let $w_1,\ldots,w_N$ be defined as in (\ref{wi}), and let
\begin{eqnarray*}
\hat{w}_i=w_{\sigma(i)}^m,\qquad \mathrm{for}\ i\in[N].
\end{eqnarray*}
Let $k$, $m$ be fixed as $N\rightarrow\infty$, and assume that each one of $u_1,\ldots,u_k$ is in a fixed compact neighborhood of $x_1,\ldots,x_k$, respectively. Let $f$ be the function defined as in (\ref{fy}), then
\begin{eqnarray*}
R_1(w,f)=\sum_{i=1}^{N}\left|\hat{w}_{i}-f\left(\frac{i}{N}\right)\right|\qquad\qquad R_{\infty}(w,f)=\sup_{1\leq i\leq N}\left|\hat{w}_{i}-f\left(\frac{i}{N}\right)\right|
\end{eqnarray*}
satisfy that $R_{\infty}(w,f)$ is bounded and $\frac{R_1(w,f)}{N}\sim O\left(\frac{1}{N}\right)\rightarrow 0$ as $N\rightarrow\infty$.
\end{lemma}

Define
\begin{eqnarray*}
F(\xi;f)=\int_0^1\log\left[\xi-f(t)\right]dt\qquad \xi\in \mathbb{C}\setminus \{f(t)| t\in [0,1]\}.
\end{eqnarray*}
Then by (\ref{sds}) we obtain
\begin{eqnarray}
\frac{s_{\lambda}(W)}{s_{\lambda^{(m)}}(W)}=\frac{1}{2\pi\mathbf{i}}\oint_{C} e^{N(y\log \xi-F(\xi, f))}\cdot Q(\xi,\lambda, f)d\xi\label{qlm};
\end{eqnarray}
where
\begin{eqnarray}
Q(\xi,\lambda, f)=\frac{e^{NF(\xi, f)}}{\prod_{i=1}^{N}(\xi-w_i^m)}.\label{qx}
\end{eqnarray}

\begin{lemma}\label{ll27}Let $f$ be defined as in (\ref{fy}). Let $k$, $m$ be fixed as $N\rightarrow\infty$, and assume that each one of $u_1,\ldots,u_k$ is in a fixed compact neighborhood of $x_1,\ldots,x_k$, respectively. Let $A$ be the smallest connected, convex region in $\CC$ containing all the points $\{f(t):0\leq t\leq 1\}$ and $\{\hat{w}_i:1\leq i\leq N\}$. Then, by Lemma \ref{ap66} as $N\rightarrow\infty$,
\begin{eqnarray*}
\log |Q(\xi,\lambda, f)|\leq O(1)\left(1+\sup_{a\in A}|\ln(\xi-a)|+\sup_{a\in A}\left|\frac{1}{\xi-a}\right|\right).
\end{eqnarray*}
\end{lemma}
\begin{proof}Let $\mathrm{Re}[z]$ denote the real part of a complex number $z\in \CC$.  From (\ref{qx}) we obtain
\begin{eqnarray*}
\log |Q(\xi,\lambda, f)|&=&\mathrm{Re}\left[\log Q(\xi,\lambda,f)\right]\\
&= & \mathrm{Re}\left[NF(\xi,f)-\sum_{i=1}^N\log (\xi-w_i^m)\right]\\
&\leq & \left|NF(\xi,f)-\sum_{i=1}^N\log (\xi-\hat{w}_i)\right|\\
&\leq & T_1+T_2.
\end{eqnarray*}
where
\begin{eqnarray*}
T_1&=&\left|\sum_{j=1}^{N}\log\left(\xi-\hat{w}_j\right)-\sum_{j=1}^{N}\log\left(\xi-f\left(\frac{j}{N}\right)\right)\right|.
\end{eqnarray*}
and 
\begin{eqnarray*}
T_2&=&\left|NF(\xi,f)-\sum_{j=1}^{N}\log\left(\xi-f\left(\frac{j}{N}\right)\right)\right|.
\end{eqnarray*}

The rest of the proof is devoted to give upper bounds to $T_1$ and $T_2$. Note that
\begin{eqnarray*}
T_1\leq \sum_{j=1}^{N}\left|\int_{\left[\hat{w}_j,f\left(\frac{j}{N}\right)\right]}\frac{dt}{\xi-t}\right|\leq \sup_{a\in A}\left|\frac{1}{\xi-a}\right|\left(\sum_{j=1}^{N}\left|f\left(\frac{j}{N}\right)-\hat{w}_j\right|\right).
\end{eqnarray*}
where $\left[\hat{w}_j,f\left(\frac{j}{N}\right)\right]$ is the line segment in $A$ starting from $\hat{w}_j$ and ending at $f\left(\frac{j}{N}\right)$. By Lemma \ref{ap66} we have
\begin{eqnarray*}
T_1\leq \sup_{a\in A}\left|\frac{1}{\xi-a}\right|O\left(1\right).
\end{eqnarray*}
Moreover,
\begin{eqnarray*}
T_2&=&\left|\sum_{j=1}^{N}\log\left(\xi-f\left(\frac{j}{N}\right)\right)-N\int_0^1\log\left(\xi-f\left(t\right)\right)dt\right|\\
&\leq&N\sum_{j=1}^{N}\sup_{\frac{j-1}{N}\leq t,s\leq \frac{j}{N}}\left|\frac{\log(\xi-f(t))-\log(\xi-f(s))}{N}\right|\\
&\leq& O(1)\left(1+\sup_{a\in A}\log|\xi-a|\right),
\end{eqnarray*}
where the last inequality follows from the definition of $f$ as in (\ref{fy}). Then the lemma follows.
\end{proof}

We shall analyze the asymptotics of the integral (\ref{sds}) by the steepest descent method; see also \cite{er56,cet65}. We will deform the contour to pass through the critical point of $y
\log \xi-F(\xi;f)$. The critical point satisfies the equation
\begin{eqnarray*}
0=\frac{d[y\log \xi-F(\xi;f)]}{d\xi}=\frac{y}{\xi}-\int_0^1\frac{dt}{\xi-f(t)}.
\end{eqnarray*}

\begin{lemma}\label{l28}Assume $f(t)>0$ for $t\in [0,1]$ and is decreasing in [0,1]. Then for any $y\in \RR\setminus \{1\}$, there exists a unique $\xi_0\in \RR$, such that 
\begin{eqnarray}
y-\int_0^1\frac{\xi_0}{\xi_0-f(t)}dt=0.\label{yme}
\end{eqnarray}
Moreover,  when $y>1$, $\xi_0>f(0)$.
%\begin{enumerate}
%\item
%\item When $y<1$, $\xi_0<f(1)$.
%\end{enumerate}
\end{lemma}
\begin{proof}Let
\begin{eqnarray*}
g(\xi):=y-\int_0^1\frac{\xi}{\xi-f(t)}dt.
\end{eqnarray*}
Then
\begin{eqnarray*}
g'(\xi)=\int_0^1\frac{f(t)}{[\xi-f(t)]^2}dt>0
\end{eqnarray*}
under the assumption that $f(t)>0$ for $t\in [0,1]$. 

Note also that 
\begin{eqnarray*}
\lim_{\xi\rightarrow -\infty}g(\xi)=\lim_{\xi\rightarrow\infty}g(\xi)=y-1;
\end{eqnarray*}
when $\xi$ increases from $-\infty$ to $f(1)$, $g(\xi)$ increases from $y-1$ to $+\infty$, and when $\xi$ increases from $f(0)$ to $\infty$, $g(\xi)$ increases from $-\infty$ to $y-1$. Therefore for any $y\neq 1$, there exists a unique $\xi\in (-\infty, f(1))\cup (f(0),\infty)$, such that the identity (\ref{yme}) holds.

It is straight forward to check that when $y>1$, $\xi_0>f(0)$. Then the lemma follows.
\end{proof}

%\begin{lemma}Let $\lambda\in \GT_N^+$ differ from $\lambda^{(m)}$ only in the first component; that is
%\begin{eqnarray*}
%\lambda:=(\lambda_1,(m-1)(N-2),(m-1)(N-3),\ldots,m-1,0).
%\end{eqnarray*}
%Let $y\in \RR$ be given by (\ref{dy}). Assume $y>1$. and let $\xi_0\in \RR$ be the unique real solution for (\ref{yme}).  Let $C'$ be the counterclockwise circle centered at 0 and passing through $\xi_0$. Then
%\begin{eqnarray*}
%\frac{s_{\lambda(W)}}{s_{\lambda^{(m)}}(W)}=\frac{1}{2\pi \mathbf{i}}\oint_{C'} \frac{\xi^{Ny}}{\prod_{i=1}^{N}(\xi-w_i^m)}d\xi
%\end{eqnarray*}
%\end{lemma}

When $y> 1$, let $\xi_0$ be the unique real solution such that (\ref{yme}) holds. Let $C'$ be the counterclockwise circle centered at 0 and passing through $\xi_0$. By Lemma \ref{l28}, $C'$ encloses all the singularities of the integrand of (\ref{sds}), since all the singularities lies on the interval $[f(0),f(1)]$ of the real line; see Figure \ref{fig:sdm}.
\begin{figure}
  \centering
  \includegraphics[width=7cm]{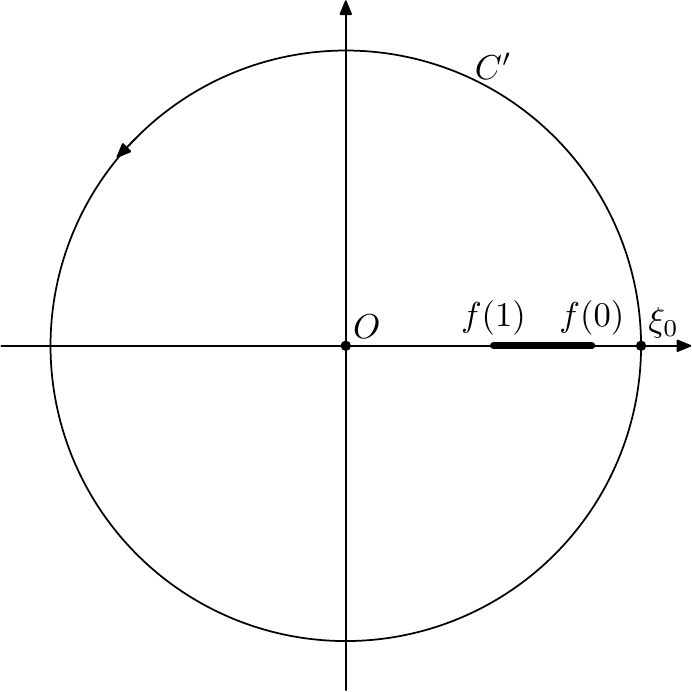}
  \caption{}
    \label{fig:sdm}
\end{figure}

Hence we have
\begin{eqnarray}
\oint_{C} \frac{\xi^{Ny}}{\prod_{i=1}^{N}(\xi-w_i^m)}d\xi=\oint_{C'} \frac{\xi^{Ny}}{\prod_{i=1}^{N}(\xi-w_i^m)}d\xi.\label{cep}
\end{eqnarray}
Moreover, for $\xi\in C'$ and $\xi\neq \xi_0$, we have
\begin{eqnarray*}
&&\mathrm{Re}[y\log\xi-F(\xi;f)]=y\log|\xi|-\int_0^1\log |\xi-f(t)|dt\\
&<& y\log\xi_0-\int_0^{1}\log |\xi_0-f(t)|dt= \mathrm{Re}[y\log\xi_0-F(\xi_0;f)].
\end{eqnarray*}
See Figure \ref{fig:sdm} for why $ \log|\xi-f(t)|>\log|\xi_0-f(t)|$ for all $t\in[0,1]$. We have the following lemma.
\begin{lemma}\label{ll29}Let $C'$ be the counterclockwise circle centered at 0 and passing through $\xi_0$. Let $\xi\in C'$. Let $\log$ denote the branch of the complex logarithmic function defined on $\mathbb{C}\setminus (-\infty,0]$ with $\log 1=0$. Let 
\begin{eqnarray}
\eta&=&\log \xi; \label{etx}\\
\eta_0&=&\log \xi_0. \label{etx0}
\end{eqnarray}
When $|\eta-\eta_0|\geq\epsilon>0$, $|\mathrm{Re}[F(\xi,f)-F(\xi_0,f)]|\geq \delta(\epsilon)$, where
\begin{eqnarray*}
\delta(\epsilon)=\frac{1}{2}\log\left(1+\frac{2f(1)\xi_0(1-\cos\epsilon)}{[\xi_0-f(1)]^2}\right).
\end{eqnarray*}
\end{lemma}

\begin{proof}Assume $\xi=\xi_0 e^{\mathbf{i}\theta}$, where $\theta=-\mathbf{i}(\eta-\eta_0)\in(-\pi,\pi)$. Then we have
\begin{eqnarray*}
\left|\mathrm{Re}[F(\xi,f)-F(\xi_0,f)]\right|&=&\left|\int_0^1\log\left|\frac{\xi_0\cos\theta-f(t)+\mathbf{i}\xi_0\sin\theta}{\xi_0-f(t)}\right|dt\right|\\
&=&\frac{1}{2}\left|\int_0^1\log\left|\frac{[\xi_0\cos\theta-f(t)]^2+[\xi_0\sin\theta]^2}{[\xi_0-f(t)]^2}\right|dt\right|\\
&=&\frac{1}{2}\int_0^1\log\left(1+\frac{2\xi_0f(t)(1-\cos\theta)}{[\xi_0-f(t)]^2}\right)dt.
\end{eqnarray*}
Since $f$ is decreasing in $[0,1]$, from Figure \ref{fig:sdm} we see that
\begin{eqnarray*}
\log\left(1+\frac{2\xi_0f(t)(1-\cos\theta)}{(\xi_0-f(t))^2}\right)\geq \log\left(1+\frac{2\xi_0f(1)(1-\cos\theta)}{[\xi_0-f(1)]^2}\right),\ \forall t\in[0,1].
\end{eqnarray*}
Moreover, the cosine function is even and strictly decreasing in $[0,\pi)$, when $|\eta-\eta_0|\geq \epsilon$, we have 
\begin{eqnarray*}
\log\left(1+\frac{2\xi_0f(1)(1-\cos\theta)}{[\xi_0-f(1)]^2}\right)\geq \frac{1}{2}\log\left(1+\frac{2f(1)\xi_0(1-\cos\epsilon)}{[\xi_0-f(1)]^2}\right).
\end{eqnarray*}
Then the lemma follows.
\end{proof}

\begin{proposition}\label{p28}Let $\lambda\in \GT_N^+$ be defined by
\begin{eqnarray*}
\lambda=(\lambda_1,(N-2)(m-1),(N-3)(m-1),\ldots, m-1,0);
\end{eqnarray*}
where $\lambda_1\geq (N-2)(m-1)$ depends on $N$.
Assume that there exists a positive integer $N_0$, such that for any $N\geq N_0$,
\begin{eqnarray*}
\frac{\lambda_1+N-1}{mN}>1.
\end{eqnarray*}
Let $\eta_1\in \RR$ (depending on $N$) be the unique real solution for the equation
\begin{eqnarray*}
\frac{\lambda_1+N-1}{mN}=\int_0^1\frac{e^{\eta_1}}{e^{\eta_1}-f(t)}dt,
\end{eqnarray*}
then
\begin{eqnarray*}
\frac{1}{N}\log \frac{s_{\lambda}(W)}{s_{\lambda^{(m)}}(W)}=\frac{(\lambda_1+N-1)\eta_1}{mN}-F(e^{\eta_1};f)+o_N(1).
\end{eqnarray*}
where $o_N(1)\rightarrow 0$ as $N\rightarrow\infty$.
\end{proposition}

\begin{proof}

Let $\eta$ and $\eta_0$ be given by (\ref{etx}) and (\ref{etx0}), respectively. Then when $\xi\in C'$, $\eta$ lies on the vertical line passing through $\eta_0$. We consider the Taylor expansion of $y\eta-F(e^{\eta},f)$ when $\eta$ is in a neighborhood of $\eta_0$, we obtain
\begin{eqnarray*}
y\eta-F(e^{\eta},f)=y\eta_0-F(e^{\eta_0},f)-\frac{(\eta-\eta_0)^2}{2}\frac{\partial^2 F(e^{\eta_0};f)}{\partial \eta^2}+(\eta-\eta_0)^3 B.
\end{eqnarray*}
where $|B|$ is bounded by the maximum of $\left|\frac{\partial^3F(e^{\eta},f)}{6\partial \eta^3}\right|$ in a neighborhood of $\eta_0$. Note that 
\begin{eqnarray*}
\frac{d[y\eta-\partial F(e^{\eta},f)]}{\partial \eta}=g(e^{\eta}).
\end{eqnarray*}
 whose value is 0 when $\eta=\eta_0$. Moreover
\begin{eqnarray*}
\frac{\partial^2 F(e^{\eta_0},f)}{\partial \eta^2}=-g'(e^{\eta_0})e^{\eta_0}<0.
\end{eqnarray*}
Let 
\begin{eqnarray*}
u&=&-\mathbf{i}\sqrt{-\frac{\partial^2 F(e^{\eta_0};f)}{\partial \eta^2}};\\
\eta&=&\eta_0+\frac{s}{u\sqrt{N}}.
\end{eqnarray*}
Making a change of variables, we obtain
\begin{eqnarray}
I:&=&\frac{1}{2\pi\mathbf{i}}\oint_{C'} e^{N(y\log \xi-F(\xi, f))}\cdot Q(\xi,\lambda, f)d\xi\\
&=&\frac{1}{2\pi\mathbf{i}}\int_{\eta_0-\mathbf{i}\pi}^{\eta_0+\mathbf{i}\pi} e^{N\left(y\eta-F(e^{\eta}, f)\right)}\cdot Q(e^{\eta},\lambda, f)e^{\eta}d\eta.\label{s1}\notag
\end{eqnarray}
We shall split the integral above into two parts
\begin{eqnarray}
I=I_1+I_2;\label{s2}
\end{eqnarray}
where
\begin{eqnarray*}
I_1=\frac{1}{2\pi\mathbf{i}}\int_{\eta_0-\mathbf{i}\epsilon}^{\eta_0+\mathbf{i}\epsilon} e^{N\left(y\eta-F(e^{\eta}, f)\right)}\cdot Q(e^{\eta},\lambda, f)e^{\eta}d\eta;
\end{eqnarray*}
and
\begin{eqnarray*}
I_2=\frac{1}{2\pi\mathbf{i}}\int_{[\eta_0-\mathbf{i}\pi,\eta_0-\mathbf{i}\epsilon]\cup[\eta_0+\mathbf{i}\epsilon,\eta_0+\mathbf{i}\pi]}e^{N\left(y\eta-F(e^{\eta}, f)\right)}\cdot Q(e^{\eta},\lambda, f)e^{\eta}d\eta.
\end{eqnarray*}
Note that
\begin{eqnarray}
|I_2|\leq \frac{1}{2\pi}\int_{[\eta_0-\mathbf{i}\pi,\eta_0-\mathbf{i}\epsilon]\cup[\eta_0+\mathbf{i}\epsilon,\eta_0+\mathbf{i}\pi]}e^{N\mathrm{Re}\left[y\eta-F(e^{\eta}, f)\right]}\cdot \left|Q(e^{\eta},\lambda, f)\right|\left|e^{\eta}\right|\left|d\eta\right|.\label{s3}
\end{eqnarray}
Let $L$ be the vertical line segment between $\eta_0-\mathbf{i}\pi$ and  $\eta_0+\mathbf{i}\pi$. By Lemmas \ref{ll27} and \ref{ll29}, we have
\begin{eqnarray*}
|I_2|\leq e^{\eta_0}e^{O(1)\left(1+\sup_{a\in A, \eta\in L}|\log (e^\eta-a)|+\sup_{a\in A,\eta\in L}\left|\frac{1}{e^\eta-a}\right|\right)}e^{N(y\eta_0-F(e^{\eta_0},f))}e^{-N\delta(\epsilon)}.
\end{eqnarray*}
Moreover,
\begin{eqnarray}
I_1&=&\frac{e^{N(y\eta_0-F(e^{\eta_0}, f))}}{2\pi\mathbf{i}}\int_{\eta_0-\mathbf{i}\epsilon}^{\eta_0+\mathbf{i}\epsilon}e^{-\frac{N(\eta-\eta_0)^2}{2}\frac{\partial^2 F(e^{\eta_0};f)}{\partial \eta^2}+NB(\eta-\eta_0)^3}Q(e^{\eta},\lambda,f)e^{\eta}d\eta\label{s4}\\
&=&\frac{e^{N[y\eta_0-F(e^{\eta_0}, f)]}}{u\sqrt{N}2\pi\mathbf{i}}\int_{-\epsilon|u|\sqrt{N}}^{\epsilon |u|\sqrt{N}}e^{-\frac{s^2}{2}+\frac{s^3}{\sqrt{N}}\tilde{B}}Q\left(e^{\eta_0+\frac{s}{u\sqrt{N}}},\lambda,f\right)e^{\eta_0+\frac{s}{u\sqrt{N}}}ds,\notag
\end{eqnarray}
where
\begin{eqnarray*}
|\tilde{B}|\leq |u|^{-3}\sup_{\eta\in[\eta_0-\mathbf{i}\pi, \eta_0+\mathbf{i}\pi]}\left|\frac{\partial^3F(e^{\eta},f)}{6\partial \eta^3}\right|.
\end{eqnarray*}

Assume
\begin{eqnarray*}
\epsilon\sim N^{-\alpha},\ \mathrm{where}\ \frac{1}{3}<\alpha<\frac{1}{2},
\end{eqnarray*}
as $N\rightarrow\infty$. Then
\begin{eqnarray*}
|u|\epsilon\sqrt{N}\sim N^{\frac{1}{2}-\alpha}\rightarrow \infty,\qquad \mathrm{as}\ N\rightarrow\infty;
\end{eqnarray*}
\begin{eqnarray*}
\sup_{s\in[-\epsilon|u|\sqrt{N},\epsilon|u|\sqrt{N}]}\left|\frac{s^3}{\sqrt{N}}\right|\leq O\left(N^{1-3\alpha}\right)\rightarrow 0,\qquad\mathrm{as}\ N\rightarrow\infty;
\end{eqnarray*}
and
\begin{eqnarray*}
N\delta(\epsilon)\sim N^{1-2\alpha}\rightarrow\infty, \qquad\mathrm{as}\ N\rightarrow\infty.
\end{eqnarray*}
Then by (\ref{qlm}), (\ref{cep}), (\ref{s1}) and (\ref{s2}), we obtain
\begin{eqnarray*}
\frac{1}{N}\log \frac{s_{\lambda}(W)}{s_{\lambda^{(m)}}(W)}=\frac{1}{N}\log\left(I_1+I_2\right).
\end{eqnarray*}
By (\ref{s3}) and (\ref{s4}), we have
\begin{eqnarray*}
I_1+I_2 = e^{N[y\eta_0-F(e^{\eta_0},f)]}\left(\frac{C_1}{\sqrt{N}}+O\left(e^{-N^{1-2\alpha}}\right)\right),
\end{eqnarray*}
where $C_1>0$ is a constant independent of $N$. Then the proposition follows.
\end{proof}

\begin{proposition}\label{pp1}Let $\lambda=(\lambda_1,\ldots,\lambda_N)\in \GT_N^+$. Assume for all the $2\leq j\leq N$, $\lambda_j=(m-1)(N-j)$.  Then
\begin{eqnarray*}
\lim_{N\rightarrow\infty}\frac{1}{N}\log\frac{s_{\lambda}(W)}{s_{\lambda^{(m)}}(W)}+
\lim_{N\rightarrow\infty}\frac{1}{N}\log\frac{s_{\lambda^{(m)}}(X)}{s_{\lambda}(X)}=0.
\end{eqnarray*}
\end{proposition}
\begin{proof}By Proposition \ref{p28}, it suffices to show that $f_W(t)=f_X(t)$. But this is obviously true by Assumption \ref{ap66} and the definitions of $X$ and $W$.
\end{proof}

Then Theorem \ref{t17} follows from Proposition \ref{pp1} and Lemma \ref{l31}.

\bigskip
\ACKNO{ZL's research is supported by National Science Foundation grant 1608896 and Simons collaboration grant 638143. ZL thanks anonymous reviewers for careful reading of the paper.}

\bibliography{fpmm}
\bibliographystyle{amsplain}
\end{document}